\documentclass[10pt]{article}
\usepackage{amsfonts,amsmath,amsthm,amssymb,fullpage,color}

\numberwithin{equation}{section}

\newtheorem{thm}{Theorem}[section]
\newtheorem{prop}[thm]{Proposition}
\newtheorem{question}[thm]{Question}

\newtheorem{cor}[thm]{Corollary}

\newtheorem{defin}[thm]{Definition}

\newtheorem{lemma}[thm]{Lemma}

\newcommand{\lk}{\mathfrak{k}}
\newcommand{\p}{\mathfrak{p}}
\newcommand{\g}{\mathfrak{g}}
\newcommand{\gc}{\g ^\mathbb C}
\newcommand{\lu}{\mathfrak u}
\newcommand{\q}{\mathfrak q}
\newcommand{\N}{\mathfrak N}

\newcommand{\actingon}{\circlearrowright}
\newcommand{\GC}{G^\mathbb C}
\newcommand{\GCv}{\GC \cdot v}
\newcommand{\Gv}{G \cdot v}
\newcommand{\VC}{V^\mathbb C}
\newcommand{\rpv}{\mathbb{P}V}
\newcommand{\rpvc}{\mathbb{RP}V^\mathbb C}
\newcommand{\cpvc}{\mathbb{CP}(V^\mathbb C)}

\newcommand{\nms}{||m||^2}
\newcommand{\nns}{||n||^2}
\newcommand{\nmus}{||\mu^*||^2}

\newcommand{\restrictto}[2]{\left. #1 \right|_{#2}}
\newcommand{\ddtat}{\restrictto{\frac{d}{dt}}{t=0}}

\begin{document}

\title{Distinguished Orbits of Reductive Groups}
\author{M. Jablonski}
\date{}
\maketitle

\begin{abstract}
We prove a generalization and give a new proof of a theorem of
Borel-Harish-Chandra on closed orbits of linear actions of reductive groups.
 Consider a real reductive algebraic
group $G$ acting linearly and rationally on a real vector space $V$. $G$ can be viewed as the real
points of a complex reductive group $G^\mathbb C$  which acts on
$V^\mathbb C := V\otimes \mathbb C$. In \cite{BHC} it was shown that
$G^\mathbb C \cdot v \cap V$ is a finite union of $G$-orbits; moreover,
$G^\mathbb C \cdot v$ is closed if and only if $G\cdot v$ is closed, see  \cite{RichSlow}.
We show that the same result holds not just for closed orbits but for
the so-called distinguished orbits.  An orbit is called distinguished if it contains a critical point of the norm squared of the moment map on projective space.  Our main result compares the complex and real settings to show $G\cdot v$ is distinguished
if and only if $G^\mathbb C \cdot v$ is distinguished.

In addition, we show that if an orbit is distinguished, then under the negative gradient flow of the norm squared of the moment map the entire $G$-orbit collapses to a single $K$-orbit.  This result holds in both the complex and real settings.

We finish with some applications  to the study of the left-invariant geometry of Lie groups.
\end{abstract}

\section{Introduction}
An analytical approach to finding closed orbits in the complex
setting was developed by Kempf-Ness \cite{Kempf-Ness} and extended to the
real setting by Richardson-Slodowy \cite{RichSlow}.  From their
perspective, the closed orbits are those that contain the zeros of
the so-called moment map. However, one can consider more generally
critical points of this moment map on projective space.  Work on the moment map in the complex setting has been done by Ness \cite{Ness} and Kirwan \cite{Kirwan}.  Following those works, the real moment map was explored in \cite{Marian} and \cite{EberleinJablo}.

Consider a real linear reductive group $G$ acting linearly and rationally on a real vector space $V$.  There is a complex linear reductive group $\GC$ such that $G$ is a finite index subgroup of the real points of $G^\mathbb C$; moreover, $\GC$ acts on the complexifcation $\VC$ of $V$.  The linear action of $G$, respectively $\GC$, extends to an action on real projective space $\rpv$, respectively complex projective space $\cpvc$.  For $v\in V$, we call an
orbit $G\cdot v$, or $G\cdot [v]$, \textit{distinguished} if the orbit $G\cdot [v]$ in real projective space contains a critical point of $\nms$, the norm square of the real moment map.  Similarly, for $v\in \VC$, we call an orbit $\GC \cdot v$, or $G\cdot \pi [v]$, \textit{distinguished} if the orbit $\GC \cdot \pi [v]$ in complex projective space contains a critical point of $\nmus$, the norm square of the complex moment map.  Here $\pi : \rpvc \to \cpvc$ is the natural projection. Our main theorems are\\

\textbf{Theorem \ref{distinguished orbits thm}}\textit{
Given $G \actingon V$, $\GC \actingon \VC$, and $[v]\in \rpv$ we have
    \begin{quote} $G\cdot [v]$ is a distinguished orbit in $\rpv$ if
    and only if $\GC \cdot \pi [v]$ is a distinguished orbit in $\cpvc$.
    \end{quote} Here $\pi : \rpv \subseteq \rpvc  \to \cpvc$ is the usual projection.}\\

\textbf{Theorem \ref{complex gradient flow thm}}\textit{ For $x\in \cpvc$, suppose $\GC \cdot x \subseteq \cpvc$ contains a critical point of $\nmus$.  If $z \in \mathfrak C \subseteq \cpvc$ is such a critical point, then $\mathfrak C \cap \GC \cdot x = U\cdot z$.
Moreover, $U\cdot z = \displaystyle{\bigcup_{g\in G^\mathbb C} } \omega (gx)$.}\\

\textbf{Theorem \ref{thm: real gradient flow thm}}\textit{ For $x\in \rpv$, suppose $G\cdot x \subseteq \rpv$ contains a critical point of $\nms$.  If $z\in \mathfrak C_\mathbb R \subseteq \rpv$ is such a critical point, then $\mathfrak C_\mathbb R \cap G\cdot x = K\cdot z$.  Moreover, $K\cdot z = \displaystyle \bigcup_{g\in G} \omega (gx)$.}\\

Here $\mu^*$ is the moment map for the action of $G^\mathbb C$ on $\cpvc$ and $\mathfrak C$ is the set of critical points of $||\mu^*||^2$ in $\cpvc$, while $m$ is the moment map for the action of $G$ on $\rpv$ and $\mathfrak C_\mathbb R$ is the set of critical points of $\nms$ in $\rpv$.  The fact that $\mathfrak C \cap G^\mathbb C\cdot x =U\cdot z$ was proven in \cite{Ness} in the complex setting; the fact that $\mathfrak C_\mathbb R \cap G\cdot x = K\cdot z$ was proven in \cite{Marian} in the real setting.  The fact that the orbit collapses under the negative gradient flow of $||\mu^*|||^2$, respectively $\nms$, to a single $U$-orbit, respectively $K$-orbit, is our new contribution (see Definition \ref{def: omega limit set} for the definition of the $\omega$-limit set).

The value of Theorem \ref{distinguished orbits thm} is as follows.  Since $G^\mathbb C \cdot v
\cap V$ is a finite union of $G$-orbits, if we can show that one of
these $G$-orbits is distinguished then all of them are.  This has
been applied to the problem of finding generic 2-step nilpotent Lie
groups which admit soliton metrics.  See
\cite{Jablo:Thesis} for more information on the
soliton problem.

These theorems are then applied in Section 6 to questions regarding the left-invariant geometry of nilpotent and solvable Lie groups.\\

\textit{Acknowledgements.}  This note is a portion of my thesis work completed under the direction of Pat Eberlein at the University of North Carolina, Chapel Hill.  I am grateful to him for suggesting many improvements to the first version of this paper.


\section{Notation and Technical Preliminaries}
Our goal is to study closed reductive subgroups $G$ of $GL(E)$ which are more or
less algebraic.  Here $E$ is a real vector space and we denote its complexification by $E^\mathbb C=E\otimes \mathbb C$.  We call a subgroup $H $ of $GL(E)$ a \textit{real algebraic group} if $H$ is the zero set of polynomials on $GL(E)$ with real coefficients; that is, polynomials in $\mathbb R [ GL(E)]$.

Consider a closed subgroup $H\subseteq GL(E)$ with finitely many connected components and its Lie algebra $\mathfrak h \subseteq \mathfrak {gl}(E)$.  Let $\mathfrak z$ denote the center of $\mathfrak h$.  We say that $H$, or $\mathfrak h =L(H)$, is \textit{reductive} if $\mathfrak h = [\mathfrak h, \mathfrak h] \oplus \mathfrak z$, $[\mathfrak h,\mathfrak h]$ is semi-simple, and $\mathfrak z \subseteq \mathfrak{gl}(E)$ consists of semi-simple endomorphisms.  Reductive groups in this sense are precisely the groups that are completely reducible, see \cite[section 1.2]{BHC}.

We say that a group $G \subset GL(E)$ is a \textit{real linear reductive group} if $G$ is a finite index subgroup of a real algebraic reductive group $H$; that is, $G$ satisfies $H_0 \subseteq G \subseteq H$, where $H_0$ is the Hausdorff identity component of $H$.  For complex algebraic groups the Hausdorff and Zariski identity components coincide.  However, this need not be true for real algebraic groups.  It is well-known that there exists a complex (algebraic) reductive group $G^\mathbb C$ defined over $\mathbb R$ such that $G$ is Zariski dense in $\GC$ and is a finite index subgroup of the real points $\GC (\mathbb R) := \GC \cap GL(E)$ of $\GC$; that is, $\GC(\mathbb R) _0 \subseteq G\subseteq \GC (\mathbb R)$.  For completeness we construct this group.

Consider $G,H$ as above.  The ideal of polynomials that describe $H \subseteq GL(E)$  also describe a variety $\overline H \subseteq GL(E^\mathbb C)$ which is defined over $\mathbb R$.  This variety $\overline H$ is the Zariski closure of $H$ in $GL(E^\mathbb C)$.  As $H$ is a subgroup of $GL(E^\mathbb C)$, it follows that $\overline H$ is actually a complex algebraic subgroup of $GL(E^\mathbb C)$,  see \cite[I.2.1]{Borel:LinAlgGrps}.  Moreover, $\overline H$ is smooth and we have  $\dim _\mathbb C \overline H = \dim _\mathbb R H$, see \cite{Whitney}.  By comparing the dimensions of these groups and their tangent spaces at the identity, one sees $L(\overline H) = L(H)\otimes \mathbb C$.  Thus, $\overline H$ is reductive as $H$ is reductive.

To construct $\GC$ we consider $\overline H_0$, the Hausdorff identity component of $\overline H$.  Recall that $\overline H_0$ is an algebraic group as the Hausdorff and Zariski identity components of $\overline H$ coincide, see \cite[I.1]{Borel:LinAlgGrps}.  Define $\GC = \overline H_0 \cdot G$.  This is a subgroup of $\overline H$ as $\overline H_0$ is normal in $\overline H$; moreover, as $\overline H_0$ has finite index in $\overline H$ and $\GC$ contains $\overline H_0$, it follows that $\GC$ has finite index in $\overline H$.  Equivalently, we can write $G^\mathbb C = \bigcup_n (g_n\cdot \overline H_0)$ where $\{g_n\}\subset G$ is a finite collection.  Thus $\GC$ is an algebraic group as it can be described as a union of varieties.  Additionally, we observe that each component $g_n \cdot \overline H_0$ of $G^\mathbb C$ intersects $G$ and that $G^\mathbb C$ is the Zariski closure of $G$.  The importance of this observation will be made clear when extending certain inner products on real vector spaces to their complexifications; see  Proposition \ref{prop: extend inner product on V to VC}.

We call the group $\GC$ the \textit{complexification of $G$}. 
We choose the complex subgroup $G^\mathbb C$ instead of $\overline H$ as $\overline H$ might have topological components which do not intersect $G$.
This is the complexification used by Mostow, see \cite[section 2]{Mostow:SelfAdjointGroups}.  We point out that we are always working in the usual topology and will explicitly state when we are talking about Zariski closed sets.\\

Let $V$ be a real vector space and denote its complexification by $V^\mathbb C=V\otimes \mathbb C$.  We will consider representations $\rho:G\to GL(V)$ that are the restrictions of morphisms $\rho^\mathbb C : \GC \to GL(V^\mathbb C)$ of algebraic groups.  See \cite{Borel:LinAlgGrps} for more information on algebraic groups and morphisms between them. 
We will call such a representation a \textit{rational representation} of $G$.
Note: We will denote the induced Lie algebra representation by the same letter.

\subsection*{Cartan Involutions}

Let $E$ be a finite dimensional real vector space.  A \textit{Cartan involution} of $GL(E)$ is an involution of the form $\theta(g)=(g^t)^{-1}$, where $g^t$ denotes the metric adjoint with respect to some inner product on $E$.  At the Lie algebra level this involution is $\theta (X)= -X^t$.

\begin{prop}[Mostow \cite{Mostow:SelfAdjointGroups}] There exists a Cartan involution $\theta$ of $GL(E)$ such that $G^\mathbb C(\mathbb R)$ is $\theta$-stable. \end{prop}

\begin{prop}[Borel, Proposition 13.5 \cite{BHC}] Let $\rho :G^\mathbb C(\mathbb R)\to GL(V)$ be a rational representation.  Let $\theta$ be a Cartan involution of $GL(E)$ such that $G^\mathbb C(\mathbb R)$ is $\theta$-stable.  Then there exists a Cartan involution $\theta_1$ of $GL(V)$ such that $\rho \circ \theta = \theta_1 \circ \rho$.\end{prop}

This proposition is extended in the next proposition which follows from sections 1 and 2 of \cite{RichSlow}.

\begin{prop}\label{prop: inner prod on V}  Let $G$ be defined as above and $\rho: G\to GL(V)$ a rational representation, then
    \begin{enumerate}
        \item There exists a $K$-invariant inner product on $V$ such that $G$ is
        self-adjoint; 
        hence, the Lie algebra $L(G)= \g$ is also self-adjoint.  That is, there exist Cartan involutions $\theta$, $\theta_1$ on $G$, $\rho(G)$, respectively, such that $\rho \circ \theta = \theta_1 \circ \rho$.

        \item There exist decompositions of $G$ and $\g$, called Cartan decompositions, so that
        $G = KP$ as a product of manifolds and $\g=\lk \oplus \p$.  Here
        $K=\{g\in G \ | \ \theta(g)=g\}$ is a maximal compact subgroup of $G$, $\lk = L(K)=\{X\in \g \ | \ \theta(X)=X \}$,
          $\p = \{X\in \g \ | \ \theta(X)=-X\}$, and $P = exp(\p)$.  Moreover, there exists an $Ad K$-invariant inner
        product $\langle \langle \cdot, \cdot \rangle \rangle $ on $\g$ so that $\g=\lk
         \oplus \p$ is orthogonal and, for $X\in \p$, $ad\ X$ is a symmetric transformation relative to $\langle \langle \cdot , \cdot \rangle \rangle $.

        \item The inner product $\langle \cdot, \cdot \rangle$ on
        $V$ is $K$-invariant, $\rho(X)$ are symmetric
        transformations for $X\in \p$, and $\rho(X)$ are skew-symmetric transformations for $X\in \lk$.
    \end{enumerate}
\end{prop}
The subspaces $\lk$ and $\p$ that arise in the Cartan decomposition above have the following set of relations

    $$\begin{array}{lcr} \left [ \lk ,\lk \right ] \subseteq \lk,\hspace{.5in} &
    \left [ \lk ,\p \right ] \subseteq \p,\hspace{.5in} &
    \left [ \p ,\p \right ] \subseteq \lk \end{array}$$
This is easy to see since $\lk$ and $\p$ are the $+1,-1$
eigenspaces, respectively, of the Cartan involution $\theta$.  We
point out that our $Ad \ K$-invariant inner product on $\g$
restricts to such on $\p$ as the relations above show that $\p$ is
$Ad\ K$-invariant.  Additionally, if the group $G$ were semi-simple, then up to scaling the only choice for $\langle \langle \cdot, \cdot \rangle \rangle $ would be $-B(\theta(\cdot),\cdot)$ on each simple factor of $\g$, where $B$ is the Killing form of $G$.

Our Cartan involution $\theta$ on $G$ is the restriction of a Cartan involution on $G^\mathbb C$, see \cite[2.8 and section 8]{RichSlow} and \cite{Mostow:SelfAdjointGroups}.  This gives  Cartan decompositions $\gc = \lu \oplus \q$ and $\GC = U \cdot Q$, where $U$ is a maximal compact subgroup of $\GC$, $Q = exp(\q)$, and $U\cap Q = \{1\}$.

We observe that the maximal compact groups $U$ and $K$ are related by $U=KU_0$.  To see this, it suffices to prove  $KU_0Q = UQ = \GC$ since $KU_0 \subseteq U$ and $U\cap Q = \{1 \}$.  Since $U_0Q=\overline H_0$ and $P\subseteq Q$, we obtain $KU_0 Q= KP\cdot \overline H_0 = G\cdot \overline H_0 = \overline H_0 \cdot G = \GC$.

The subspaces $\lu$, $\q \subseteq \gc$ are related to $\lk$, $\p \subseteq \g$ as follows
\begin{eqnarray*} \lu &=& \lk \oplus i\p \\
                \q &=& i\lk \oplus \p \end{eqnarray*}
These two subspaces of $L \GC =\g^\mathbb C$ have a nice interpretation relative to a particular inner product on $V^\mathbb C$.  Our construction of this inner product on $V^\mathbb C$ is similar to that done in sections 2 and 8 of \cite{RichSlow}.  We will be consistent with their notation.\\

\begin{prop}\label{prop: extend inner product on V to VC} The $K$-invariant inner product $<,>$ on $V$, described in Proposition \ref{prop: inner prod on V}, extends to a $U$-invariant inner product $S$ on $V^\mathbb C$ with a similar list of properties for $G^\mathbb C$.  Additionally, the inner product $\ll,\gg $ on $\g$ extends to an $Ad\ U$-invariant inner product $\mathbb S$ on $\g^\mathbb C$.\end{prop}

\begin{proof}  The proof of this fact follows the construction of $S$ in A2 (proof of 2.9) in \cite{RichSlow}.  Define the inner product on $V^\mathbb C $ as
$$S(v_1+i\ v_2, w_1+i\ w_2) = <v_1,w_1>+<v_2,w_2>$$
In this way, $V$ and $iV$ are orthogonal under $S$ and $i$ acts as a skew-symmetric transformation on $V^\mathbb C$ relative to $S$.  $S$ is positive definite on $V^\mathbb C$.

Recall that $U=KU_0$ (see the remark above), and observe that $S$ is $K$-invariant as $K$ preserves $V$, $iV$ and $<,>$ is $K$-invariant.  Thus to show $U$-invariance, once just needs to show $U_0$-invariance.  This follows since $\rho (\lu)$ acts skew-symmetrically and $U_0 = exp(\lu)$.

We leave to the reader the details of showing that $\rho (\lu)$ acts skew-symmetrically and  $\rho (\q)$ acts symmetrically relative to $S$.  Lastly, the extension of $<<,>>$ on $\g$ to $\mathbb S$ on $\gc$ is a special case of the above work.
\end{proof}

We say that the inner products on our complex spaces are \textit{compatible} with the inner products on the underlying real spaces.  The inner product $S$ constructed here gives rise to a $U$-invariant Hermitian form $H=S+iA$ on $V^\mathbb C$ where we define $A(x,y)=S(x,iy)$.  This Hermitian form is compatible with the real structure $V$ in the sense of Richardson and Slodowy, that is, $A=0$ when restricted to $V\times V$; see sections 2 and 8 of \cite{RichSlow}.

\subsection*{Moment maps}
Next we define our moment maps.  The motivation for these
definitions comes from symplectic geometry and the actions of compact
groups on compact symplectic manifolds.  In the complex setting,
this moment map coincides with the one from the symplectic structure
on $\cpvc$.  For more information see \cite{Ness} and
\cite{Guillemin-Sternberg}.\\

\textbf{Real moment maps}. Given $G\circlearrowright V$ we define $\tilde{m} :V \to \p$ implicitly by
    $$\ll \tilde m (v), X \gg \ = \ <Xv,v>$$
for all $X\in \p$.  Notice that $\tilde m (v)$ is a real homogeneous
polynomial of degree 2.  Equivalently, we  could define
$\tilde m :V\to \g$; then using $K$-invariance and $\lk \perp \p$ we
obtain $\tilde m (V) \subseteq \p$.

We can just as
well do this for $\GC \actingon \VC$ where we regard $\GC$ as a real
Lie group.  We use the inner product $\mathbb S$ on $\VC$. The (real) moment map for $\GC \actingon \VC$, denoted by
$\tilde n:\VC \to \q$, is defined by
    $$\mathbb S (\tilde n (v), Y) = S(Yv,v)$$
for $Y\in \q$ and $v\in \VC$.

Since these polynomials are homogeneous, they give rise to well defined maps on (real) projective space.  Define
\[
\begin{array}{ll}
     m  :  \rpv \to \p  \ \ \ \ \ \ \ &  n: \rpvc \to \q \\
    m[v] =\tilde m(\frac{v}{|v|})=\frac{\tilde m(v)}{|v|^2} \ \ \ \
    \ \ \ \
    & n[w] = \tilde n(\frac{w}{|w|})=\frac{\tilde n(w)}{|w|^2}
\end{array} \]
where $|w|^2=S(w,w)$ and $S=<,>$ on $V$.  Since
$V\subseteq \VC$ we have $\rpv \subseteq \rpvc$; this is our main
reason for studying the real moment map on $\GC$.  The next lemma compares these two real moment maps.

\begin{lemma}\label{lemma: n|pv = m} $n$ restricted to $\rpv$ equals $m$.
\end{lemma}

\begin{proof}Recall that $n$ takes values in $\q=i\lk \oplus \p$ and $m$ takes values in $\p \subseteq \q$.  Take $v\in V$ and $X\in \lk$ then
    $$\mathbb S (\tilde n(v),iX)=S(iX\cdot v,v)=0$$
as $V \perp iV$ (see Proposition \ref{prop: extend inner product on V to VC} 
), and we
are using $(iX)\cdot v = i (X\cdot v)$,  i.e., $\gc$ acts $\mathbb
C$-linearly on $\VC$.  Since $\g \perp i\g$ under $\mathbb S$, we
have $i\lk \perp \p$.  Thus $ \tilde n(v) \in \p \subseteq \q$.
Now take $X\in \p$.

\[ \begin{array}{clll}

    \mathbb S(\tilde n(v),X) &=& \ll \tilde n(v), X \gg  & \mbox{
        by compatibility of $\g \subseteq \g^\mathbb C$}\\
    || \\
    S(Xv,v) &=& <Xv,v> & \mbox{ by compatibility of $V\subseteq \VC$}\\
            &=& \ll \tilde m(v),X \gg & \mbox{ by
            definition/construction of $\tilde m$ }

\end{array} \]
Therefore, $ \tilde n(v)=\tilde m(v)$ for $v\in V\subseteq \VC$, which
implies $n[v]=m[v]$ for $[v] \in \rpv \subseteq \rpvc$.
\end{proof}

\textbf{Complex moment maps}.  We choose a notation that is similar to Ness \cite{Ness} as we are following her definitions; the only difference is that we use $\mu$ where she uses $m$.  For $v\in \VC$, consider $\rho_v:G^\mathbb C \to \mathbb R$ defined by  $\rho_v(g)= |g\cdot v|^2$, where $|w|^2= H(w,w)= S(w,w)$.  Define a map $\mu:\cpvc \to \q^* = \mbox{Hom} (\q, \mathbb R)$ by $\mu (x)=\frac{d\rho_v(e)}{|v|^2}$, where $v\in \VC$ sits over $x\in \cpvc$,  cf. \cite[section 1]{Ness}.  We define the complex moment map $\mu^*:\cpvc \to \q$ by $\mu = \mathbb S (\mu^*, \cdot)$.  Note, taking the norm square of our complex moment map will give us the norm square of the moment map in Kirwan's setting; in Kirwan's language $i\mu$ would be the moment map \cite[section 1]{Ness}.\\

Let $\pi$ denote the projection $\pi: \rpvc \to \cpvc$.

\begin{lemma}\label{lemma:mu circ pi=2n} The complex and real moment maps for $G^\mathbb C$ are related by $\mu^* \circ \pi = 2n$
\end{lemma}
\begin{proof} Many of our computations have the same flavor as those of Ness; we employ her ideas.  Take an orthonormal basis $\{ \alpha_i \}$ of $i\lu=\q$ under $\mathbb S$.  Also let $x=\pi [v] \in \cpvc$ for $v\in \VC$.  Then
    \begin{eqnarray*}
    \mu ^*(x) &=& \sum_i \mathbb S(\mu^*(x),\alpha_i)\alpha_i\\
    &=& \sum_i [\mu (x)\alpha_i] \alpha _i\\
    &=& \sum_i \frac{1}{||v||^2} d\rho _v (e)(\alpha _i)\alpha _i \\
    &=& \sum_i \frac{1}{||v||^2} \ddtat \ ||exp \
    t\alpha_i \cdot v ||^2\alpha_i
    \end{eqnarray*}
Here the norm on $\VC$ is from $H=S+iA$.  But $S$ is the inner
product being used on $\VC$, and so $H(w,w)=S(w,w)$ tells us that
$\mu^* (x)$
    \begin{eqnarray*}
     &=& \sum_i \frac{1}{||v||^2} 2 S(\alpha_i v,v)
     \alpha_i \ \ \ \ \ \ \\
     &=& \sum_i 2 \ \mathbb S( \tilde n[v],\alpha_i)\alpha_i\\
     &=& 2 \tilde n[v]
    \end{eqnarray*}
\end{proof}

\textit{Remark}. Since $\rpv$ is not a subspace of $\cpvc$, we use $\rpvc$ and the real
 moment map of $G^\mathbb C$ to work between the known results of Kirwan and Ness to get information about our real group $G \actingon \rpv$.

\section{Comparison of Real and Complex cases}
Most of algebraic geometry and Geometric Invariant Theory has been worked out exclusively for fields which are algebraically closed.  We are interested in the real category and will exploit all the work that has already been done over $\mathbb C$.  We use and refer the reader to \cite{Whitney} as our main reference for real algebraic varieties.

Recall that our representation $\rho:G \to GL(V)$ is the restriction of a representation of $\GC$.  The
following is proposition 2.3 of \cite{BHC} and section 8 of \cite{RichSlow}.  Originally this was stated as a comparison between
$\GC(\mathbb R)_0$-orbits and $\GC$-orbits, however, it can be restated as a comparison between $G$ and $G^\mathbb C$ orbits, for any $G$ satisfying $\GC(\mathbb R)_0 \subseteq G \subseteq \GC(\mathbb R)$.  This is true as  $\GC(\mathbb R)_0$ has finite index in $G$.

\begin{thm}\label{BHC: real orbits from complex}  Let $v\in V$, then 
 $\GC \cdot v \cap V = \displaystyle{\bigcup_{i=1}^m} X_i$ where each $X_i$ is a  $G$-orbit.  
Moreover, $\GC \cdot v$ is closed in $V^\mathbb C$ if
and only if $G\cdot v$ is closed in $V$.
\end{thm}

Complex group orbits have some nice properties
that we don't enjoy over the real numbers.  For example, the Hausdorff and
Zariski closures of a group orbit are the same for a complex linear
algebraic group.  One property that does translate to the reals is that
the boundary of an orbit consists of orbits of
strictly lower dimension.  See section 8.3 of \cite{Humphreys} for the complex setting and see below for the real setting.  For some interesting examples of semi-simple real algebraic groups whose orbit closure is not the Zariski closure see \cite{EberleinJablo}.

\begin{prop}\label{observation on boundaries in V} Let $G$ and $\GC$ be defined as above.  Take $v\in V\subseteq \VC$. Then
\begin{enumerate}
    \item  $dim_\mathbb R G\cdot v = dim_\mathbb C \GC \cdot v$.
    \item  $\partial (G\cdot v) = \overline{G\cdot v} - G\cdot v$ consists of $G$-orbits of strictly smaller dimension.
    \item  $\GCv \cap \overline{\Gv}=\Gv.$
\end{enumerate}
\end{prop}

\begin{proof}[Proof of a] Notice that for a real Lie group $H$, $H\cdot v \simeq H/H_v$.  Let $\mathfrak h$ be the Lie algebra of H. Then at the Lie algebra level it is easy to see that $(\mathfrak h_v)^\mathbb C = (\mathfrak h^\mathbb C)_v$, for $v\in V\subseteq \VC$.  As $\dim_\mathbb R G = \dim_\mathbb C G^\mathbb C$, we are done.\\

\noindent \textit{Proof of b}. Recall two facts about complex group orbits.  First, the boundary $\overline{G^\mathbb C \cdot v} - G^\mathbb C \cdot v$ of the complex group orbit $\GC \cdot v$ consists of $\GC$-orbits of strictly smaller dimension, see \cite[section 8.3]{Humphreys}.  Second, $\GC \cdot v \bigcap V = \displaystyle{\bigcup_1^m }X_i$, where each $X_i$ is a $G$-orbit.  Moreover, each $X_i$ is closed in $\GC \cdot v \bigcap V $ as it is a finite union of connected components of $G^\mathbb C \cdot v \cap V$, see \cite[Proposition 2.3]{BHC}.
If $v\in X_i$ for $1\leq i \leq m$, then $G\cdot v = X_i$ and $\overline{G\cdot v}\cap G^\mathbb C\cdot v = X_i$.  If $w\in \overline{G\cdot v}-G\cdot v$, then $w\in \overline{G^\mathbb C\cdot v}-G^\mathbb C\cdot v$, and it follows from a and the first remark above that $G\cdot w$ has smaller dimension than $G\cdot v$.
\\

\noindent \textit{Proof of c}. This follows immediately from b and its proof.
\end{proof}

\subsection*{Orbits in Projective space}
Since our groups act linearly on vectors spaces we can consider the induced actions on projective space
    $G \circlearrowright  \mathbb P V$ and $\GC \circlearrowright
    \mathbb {RP} V^ \mathbb C$.

\begin{lemma}\label{lemma: in proj. orbits  GC closure cap G = G} For $v\in V$, $\GC \cdot [v] \cap \overline{G\cdot [v]}
    = G\cdot [v]$ in $\mathbb{RP} V^\mathbb C$.
\end{lemma}
This is the same result in projective space that we had for our
vector spaces.
\begin{proof}
The actions of $\mathbb R^* \times G$ and $G$ on $\rpv$ are the same; moreover, $(\mathbb R^* \times G)^\mathbb C = \mathbb C^*\times G^\mathbb C $.  Given $v\in V$  take
$g_n\in G$ and $g\in \GC$ such that $[g_n v]\to [g v]$ in $\rpv$.  Then we want
to show $[gv]\in G\cdot [v]$. Now take $r_n,r \in \mathbb R$ such
that $r_n g_n v, r g v$ have unit length in $V^\mathbb C$. We can
assume $r_n g_n v \to r g v$ by passing to $-r$ and a
subsequence if necessary.\\
Then $r_n g_nv \to rgv \in \mathbb C^* \times \GC \cdot v \cap
\overline{\mathbb R^* \times G\cdot v}$.  Therefore, $rgv\in \mathbb
R^* \times G\cdot v$ using Proposition \ref{observation on boundaries in V}c and our result follows.
\end{proof}

\section{Closed and Distinguished Orbits}
We begin with a theorem of Richardson and Slodowy (for real groups) which follows the work of Kempf and Ness (for complex groups).  To find which orbits are closed, one looks for the infimum of $|g\cdot v|^2$ along the orbit.  Such a vector is called a \textit{minimal vector} and  it occurs on the orbit precisely when our orbit is closed.  Let $\mathfrak M$ denote the set of minimal vectors in $V$.  Although not stated using the moment map, the following was proven in \cite[Theorem 4.4, 7.3]{RichSlow}.

\begin{thm}\label{closed orbit thm} $\Gv$ is closed if and only if there exists $w\in
\Gv$ such that $\tilde m (w)=0$.  Such a vector $w$ is  minimal.  Moreover, $\mathfrak M = \tilde m ^{-1}(0)$ and $\overline{G\cdot v}\cap \mathfrak M $ is a single $K$-orbit.  \end{thm}

Equivalently we could find the zero's of $||\tilde m ||^2$ to find
the minimal vectors.  Minimal vectors are used to understand the \textit{semi-stable points}, that is, all the vectors whose orbit closure does not contain zero.  In contrast, the \textit{null cone} is the set of vectors whose orbit closure does contain zero.  To study the null cone, we move to projective space.  Clearly we cannot use minimal vectors to study the geometry of the null cone, so instead of looking for zeros of $||m||^2$ on $\rpv$ we look for critical points of $||m||^2$.

\begin{defin}  We say that $v\in V$ or $[v]\in \rpv$ is distinguished if $||m||^2:\rpv \to \mathbb R$ has a critical point at $[v]$.  We say that an orbit $G\cdot v$ or $G\cdot [v]$ is distinguished if it contains  a distinguished point.  Analogously, we define distinguished points and $\GC$-orbits in $\VC$ and $\cpvc$ using $\nmus$.
\end{defin}

Minimal vectors are distinguished as zero is an absolute minimum of the function $||m||^2$.  Our goal is to find an analogue of Theorem \ref{BHC: real orbits from complex} for distinguished orbits. To understand critical points of $||m||^2$, we will find a way to relate this function to $||\mu^*||^2$ by means of $||n||^2$.  Recall that $||\mu^*||^2$ has been studied extensively in \cite{Ness,Kirwan}.

Our first observation is that the only closed orbits $G\cdot [v]
\subseteq \rpv$ occur when $G\cdot [v] = K\cdot [v]$.  This is well
known, but an elegant and geometric proof is easily obtained using properties of the moment map; see, e.g., \cite[Theorem 1]{Marian}.  So our main interest is in the remaining distinguished orbits.

\begin{prop}\label{critical points of \nns iff for \nms} If $[v]\in \rpv$, then $grad\ \nns [v] = grad\ \nms [v] \in T_{[v]}G\cdot [v]$.  Hence, $\nns$ has a critical point at $[v]\in \rpv \subseteq \rpvc$ if and only if $\nms$ does so.  Moreover, if $[v]\in \rpv$, and $\varphi_t [v]$ is the integral curve of $-grad\ \nns$ starting at $[v]$, then $\varphi_t[v] \in G\cdot [v] \subseteq \rpv$ for all $t$.
\end{prop}

Before proving the proposition, we study the gradients of these functions.
Let $\phi :G^\mathbb C\times V^\mathbb C \to V^\mathbb C$ denote the action of $G^\mathbb C$ on $V^\mathbb C$, and let $\phi_v : G^\mathbb C \to V^\mathbb C$ denote the induced map for every $v\in V^\mathbb C$.
 We define vector fields on $V^\mathbb C$ and $\rpvc$ as follows. On $V^\mathbb C$ we define
    $$\tilde X_\alpha (v) := d\phi_v (\alpha)=\ddtat exp \ t\alpha \cdot v $$
for $\alpha \in \gc$.  And on $\rpvc$
    $$X_\alpha [v] := \pi_* \tilde X_\alpha (v)$$
where $\pi : V^\mathbb C \to \rpvc$ is projection.  Note, this is well defined
as our action $G^\mathbb C \actingon V^\mathbb C$ is linear.

\begin{lemma}\label{lemma: grad m2= 4X m(x)} For $x\in \rpv$, $grad \ \nms (x)= 4 X_{m(x)}(x)$.  For $x\in \rpvc$, $grad \ \nns (x) = 4 X_{n(x)}(x)$. \end{lemma}

Marian proves the first statement for $\nms$ on $\rpv$, see \cite[Lemma 2]{Marian}.  Her proof carries over to obtain the statement for $\nns$ on $\rpvc$.

\begin{proof}[Proof of proposition \ref{critical points of \nns iff for \nms}]
The first assertion follows from Lemma \ref{lemma: grad m2= 4X m(x)}, Lemma \ref{lemma: n|pv = m}, and the fact that $m[v] \in \p \subseteq \g$ for $[v]\in \rpv$.  The second and third assertions follow immediately from the first.
\end{proof}

Next we relate the actions of our complex group $\GC$ on $\rpvc$ and $\cpvc$.  By Lemma \ref{lemma:mu circ pi=2n} we know $||\mu^* \circ \pi [v]||^2=4||n[v]||^2$ for $v\in  \VC$ and $\pi : \rpvc \to \cpvc$.  This shows that $||n||^2$ is not just $U$-invariant, it is also $U\times \mathbb C^*$-invariant.  We wish to relate the actions of $\GC$ on $\rpvc$ and $\cpvc$ by comparing their gradients from the natural Riemannian structures on these projective spaces.


\subsection*{The Riemannian structures and gradients on projective space}

Recall that projective space can be endowed with a natural Riemannian metric so that projection from the vector space is a Riemannian submersion.  This natural Riemannian metric is called the Fubini-Study metric and is defined as follows.  Take $\zeta_i \in T_{[w]}\mathbb {KP}(V^\mathbb C)$, where $\mathbb K=\mathbb R$ or $\mathbb C$.  Let $\Pi^\mathbb K :V^\mathbb C \to \mathbb {KP}(V^\mathbb C)$ be the usual projection and take $\xi_i \in T_w V^\mathbb C$ such that $\Pi_*^\mathbb K(\xi_i)=\zeta_i$.  The Fubini-Study metric on $\mathbb {KP}(V^\mathbb C)$ is defined by

$$(\zeta_1,\zeta_2)=\frac{(\xi_1,\xi_2)(w,w)-(\xi_1,w)(\xi_2,w)}{(w,w)}$$

One can naturally identify the tangent space $T_{\Pi^\mathbb K(w)} \mathbb {KP}(V^\mathbb C)$ with the orthogonal compliment of $\mathbb K-span<w>$ in $T_w V^\mathbb C$.  In our setting, we are using $S$, the extension of $<,>$ on $V$, as our inner product on $V^\mathbb C$.  Using these natural choices of Riemannian structures on $\rpvc$ and $\cpvc$ we see that $\pi :\rpvc \to \cpvc$ is also a Riemannian submersion.

We are interested in the negative gradient flow of the moment map.  Let $\varphi_t$ denote the negative gradient flow of $||n||^2$ on $\rpvc$ and $||\mu^*||^2$ on $\cpvc$.

\begin{defin}\label{def: omega limit set}  The $\omega$-limit set of $\varphi _t (p) \subseteq \rpvc$ is the set $\{ q\in \rpvc \ | \ \varphi_{t_n}(p) \to q \mbox{ for some sequence } t_n \to \infty \mbox{ in } \mathbb R \}$.  We denote this set by $\omega (p)$.
\end{defin}

Analogously, we can define the $\omega$-limit set of $\varphi_t(p)\subseteq \cpvc$ and we denote this set by $\omega (p)$ also.  It is easy to see that $\omega (p)$ is invariant under $\varphi_t$ for all $t$.\\

\textit{Remark}.  We observe that points in the $\omega$-limit set of a negative gradient flow are fixed points of the flow, that is, critical points of the given function.  In general this is not true for $\omega$-limit points associated to non-gradient flows.  We include a brief argument for the reader.

Consider $F:M\to \mathbb R$ and let $\varphi_t(p)$ denote the integral curve of $-grad\ F$ starting at $p\in M$.  Observe that $F$ is decreasing along $\varphi_t(p)$.  Suppose $\omega(p)$ is non-empty.  Then we can define $c=\displaystyle \lim_{t\to \infty} F(\varphi_t(p))$ to obtain $\omega(p) \subseteq F^{-1}(c)$.  Thus for $q\in \omega(p)$ we see that $\varphi_t(q) \subseteq F^{-1}(c)$.  Hence, $grad\ F (q) = 0$.  That is, points in the $\omega$-limit set of $-grad\ F$ are critical points for $F$.

\begin{prop}\label{prop: n vs. mu*, crit. pts.} Endow $\rpvc$ and $\cpvc$ with the Riemannian metrics so that the projections from $V^\mathbb C$  are Riemannian submersions.  Then the following are true for $[v]\in \rpvc$
\begin{enumerate}
    \item $4 \pi_* \ grad \ ||n||^2 [v] = grad\ ||\mu^*||^2 (\pi[v])$

    \item $[v]\in \rpvc$ is a critical point of $|| n||^2$
if and only if $\pi [v] \in \cpvc$ is a critical point of $||
\mu^*||^2$.

    \item $\varphi_t \circ \pi = \pi \circ \varphi_{4t}$, where $\varphi_t$ denotes the negative gradient flow of $||n||^2$ on $\rpvc$ or $||\mu^*||^2$ on $\cpvc$.

    \item $\pi(\omega([v]))=\omega(\pi[v])$, where $\omega (p)$ denotes the $\omega$-limit set of the negative gradient flow starting from $p$.
\end{enumerate}
\end{prop}

\begin{proof}
Applying Lemma \ref{lemma:mu circ pi=2n} we have
    \begin{eqnarray*}
    4<grad || n||^2 [v], w_{[v]}> &=& 4 \ddtat ||
     n [v+tw] ||^2 \\
    &=& \ddtat || \mu ^* \pi [v+tw] ||^2 \\
    &=& <grad ||\mu ^*||^2 (\pi [v]), \pi _* w_{[v]}>
    \end{eqnarray*}

Since $\pi_*$ is a submersion we have that $\pi_*$ maps the horizontal subspace of $T_{[v]} \rpvc $ isometrically onto $T_{\pi [v]} \cpvc$ and part $a.$ is proven.  Thus if $[v]$ is a critical point for $|| n ||^2$, then $\pi [v]$ is  one for $||\mu^*||^2$.  To obtain the reverse direction use the $\mathbb C^*$-invariance of $||n||^2$.  This proves part $b$.

Proof of part $c$.  Let $[v]\in \rpvc$.  Consider the curve $\pi \circ \varphi_{4t} [v]$ in $\cpvc$.  This curve satisfies the following differential equation
    $$\frac{d}{dt} \pi \circ \varphi_{4t}[v] = \pi_* \ 4 (-grad \ ||n||^2)(\varphi_{4t}[v])=-grad\ \nmus (\pi \circ \varphi_{4t}[v])$$
That is, the curve $\pi \circ \varphi_{4t}[v]$ is the integral curve of the negative gradient flow of $\nmus$ starting at $\pi [v]$.  Thus, $\pi \circ \varphi_{4t}=\varphi_t \circ \pi$.

Proof of part $d$.  We will show containment in both directions.  Take $p\in \omega [v]$, then there exists a sequence of $t_n \to \infty$ such that $\varphi_{t_n}[v] \to p$ in $\rpvc$.  Using part $c$, we have $\varphi_{t_n/4}(\pi [v])=\pi \circ \varphi_{t_n}[v] \to \pi(p)$.  That is, $\pi (p) \in \omega (\pi [v])$, or $\pi (\omega [v]) \subseteq \omega(\pi[v])$.  To obtain the other direction, take $q\in \omega (\pi[v])$ and $t_n \to \infty$ so that $\varphi_{t_n} (\pi[v]) \to q$ in $\cpvc$.  Consider the set $\varphi_{4t_n}[v]$ in $\rpvc$.  Since $\rpvc$ is compact, we can find a limit point of this set and passing to a subsequence we may assume $\varphi_{4t_n}[v] \to p$.  Then $p\in \omega [v]$, $\pi (p)=q$ by (c) and we have shown $q\in \pi(\omega [v])$.  That is, $\omega(\pi[v])\subseteq \pi(\omega [v])$.
\end{proof}

We finish the section by stating our main theorem and some corollaries.

\begin{thm}\label{distinguished orbits thm}  Given $G \actingon V$, $\GC \actingon \VC$, and $[v]\in \rpv$ we have
    \begin{quote} $G\cdot [v]$ is a distinguished orbit in $\rpv$ if
    and only if $\GC \cdot \pi [v]$ is a distinguished orbit in $\cpvc$.
    \end{quote}  Here $\pi : \rpv \subseteq \rpvc  \to \cpvc$ is the usual projection.
\end{thm}

\textit{Remark}.  Analysis of the proof of Theorem \ref{distinguished orbits thm} shows the following.  Given $v\in V \subseteq \VC$, the orbits $G\cdot [v] \subseteq \rpv$ and $\GC\cdot \pi [v] \subseteq \cpvc$ being distinguished is equivalent to $\GC \cdot [v] \subseteq \rpvc$ being distinguished using $\nns$ on $\rpvc$.

\begin{cor}\label{cor: real orbits which are complex equiv} Suppose we have $v_1, v_2 \in V$ with distinct
$G$-orbits but whose $\GC$-orbits are the same.  Then $G\cdot [v_1]$
is distinguished if and only if $G\cdot [v_2]$ is distinguished.\end{cor}

\textit{Remark}.  The phenomenon of two vectors having different real orbits but the same complex orbit
happens often.  This corollary was a necessary ingredient in the solution to the problem of
showing that generic 2-step nilmanifolds admit soliton metrics.
See \cite{Jablo:Thesis}.  This corollary is also used to prove other interesting geometric results; see, e.g., Theorem \ref{thm: real forms of same complex group}.

\section{Proofs of Main Theorems}
Here we prove Theorem \ref{distinguished orbits thm} on
distinguished orbits.  To do this, we first prove a statement for
complex moment maps in the complex setting.  Then we will relate the
complex moment map information to the real moment map for the $\GC$
action.

\textit{Remark}.  For $x\in \cpvc$, the critical points of $||\mu^*||^2$ restricted to $G^\mathbb C \cdot x$ are precisely the critical points of $||\mu^*||^2$ as a function on $\cpvc$.  This is because $grad \ ||\mu^*||^2 (x)$ is always tangent to $G^\mathbb C \cdot x$.  We denote the set of critical points of $\nmus$ in $\cpvc$ by $\mathfrak C$.

\begin{thm}\label{complex gradient flow thm}  For $x\in \cpvc$, suppose $\GC \cdot x \subseteq \cpvc$ contains a critical point of $\nmus$.  If $z \in \mathfrak C \subseteq \cpvc$ is such a critical point, then $\mathfrak C \cap \GC \cdot x = U\cdot z$.
Moreover, $U\cdot z = \displaystyle{\bigcup_{g\in G^\mathbb C} } \omega (gx)$.
\end{thm}

Let $\mathfrak C _\mathbb R$ denote the set of critical points of $\nms$ on $\rpv$.  We have a real analogue of the theorem above.

\begin{thm}\label{thm: real gradient flow thm} For $x\in \rpv$, suppose $G\cdot x \subseteq \rpv$ contains a critical point of $\nms$.  If $z\in \mathfrak C_\mathbb R \subseteq \rpv$ is such a critical point, then $\mathfrak C_\mathbb R \cap G\cdot x = K\cdot z$.  Moreover, $K\cdot z = \displaystyle \bigcup_{g\in G} \omega (gx)$.
\end{thm}

Before proving Theorems \ref{complex gradient flow thm} and \ref{thm: real gradient flow thm}, we apply Theorem \ref{complex gradient flow thm} to prove Theorem \ref{distinguished orbits thm}. \\

\textbf{Proof of Theorem \ref{distinguished orbits thm}.}
Suppose first that $G\cdot [v]$ is distinguished.  Then $G\cdot[v]=G\cdot [w]$ where $[w]$ is a critical point of $\nms$.  But now Proposition \ref{critical points of \nns iff for \nms} implies that $[w]$ is a critical point of $\nns$ and Proposition \ref{prop: n vs. mu*, crit. pts.} implies that $\pi [w]$ is a critical point of $\nmus$; that is, $\GC \cdot \pi [v]$ is distinguished.

Now suppose $\GC \cdot \pi[v]$ is distinguished.  Our goal is to show that the orbit $G\cdot [v]$ in $\rpv$ contains a critical
point of $\nms$.  We will use the $\GC$ action on $\rpvc$ and the real moment map of this action.
As $\GC \cdot \pi [v]$ is distinguished, and $\pi :\GC \cdot [v] \to \GC \cdot \pi [v]$ is surjective, there exists $w\in \GC \cdot [v]$ such that $\pi [w] \in \GC \cdot \pi [v]$ is a critical point of $\nmus$.

Apply the negative gradient flow of $\nns$ in $\rpvc$ starting at $[v]\in \rpv$.  By Proposition \ref{critical points of \nns iff for \nms} this is the negative gradient flow of $\nms$ and the $\omega$-limit set $\omega [v] \subseteq \overline {G\cdot [v]}$ consists of critical points of $\nns$ and $\nms$ (see the remark following Definition \ref{def: omega limit set}).  By Proposition \ref{prop: n vs. mu*, crit. pts.} d and Theorem \ref{complex gradient flow thm}, we have $\pi(\omega [v]) = \omega (\pi [v]) \subseteq U\cdot \pi [w]$; hence, $\omega [v] \subseteq \pi^{-1}(U\cdot \pi [w] )= \mathbb C^*\times U\cdot [w] \subseteq \mathbb C^*\times G^\mathbb C\cdot [v]$.  This implies

$$ \omega [v] \subseteq \mathbb C^* \times G^\mathbb C\cdot [v] \cap \overline {G\cdot [v]} \subseteq
    \mathbb C^*\times G^\mathbb C\cdot [v] \cap \overline{\mathbb R^*\times G\cdot [v]} = \mathbb R^*\times G\cdot [v] =G\cdot [v]$$
by Lemma \ref{lemma: in proj. orbits  GC closure cap G = G} and the fact that $(\mathbb R^*\times G)^\mathbb C = \mathbb C^*\times G^\mathbb C$.  Hence $\omega [v]$ consists of critical points of $\nms$ that lie in $G\cdot [v]$.  This proves Theorem \ref{distinguished orbits thm}.\\

Before proving Theorem \ref{complex gradient flow thm}, we prove Theorem \ref{thm: real gradient flow thm}.  The proof of this theorem is actually embedded in the proof of Theorem \ref{distinguished orbits thm}.  We present it here.\\

\textbf{Proof of \ref{thm: real gradient flow thm}}  The fact that $\mathfrak C_\mathbb R \cap G\cdot x$ constitutes a single $K$-orbit is the content of \cite[Theorem 1]{Marian}.  In \cite{Marian} $G$ is taken to be semi-simple; however, all the results hold for $G$ real reductive with the same proofs, mutatis mutandis.  Our original contribution is the second statement of the theorem.  We prove it here.

Suppose $G\cdot x \subseteq \rpv$ contains a critical point $z$ of $\nms$.  Then the orbit $\GC \cdot \pi (x)$ is distinguished in $\cpvc$ by Theorem \ref{distinguished orbits thm}.  The proof of Theorem \ref{distinguished orbits thm} shows, for $g\in G$, $\omega (gx)$ consists of critical points of $\nms$ in $G\cdot x$.  By Theorem 1 of \cite{Marian}, we have $\omega (gx) \subseteq K\cdot z$.  Hence,
$\displaystyle \bigcup_{g\in G}\omega (gx) = K\cdot z$, since $\omega (y)= \{ y \}$ for all $y\in K\cdot z$.\\

Lastly we have to prove Theorem \ref{complex gradient flow thm}.  The first
statement is proven in \cite[Theorem 6.2]{Ness}.  That is, the critical points of $||\mu^*||^2$ on a $\GC$-orbit comprise a single $U$-orbit.  As in Theorem \ref{thm: real gradient flow thm}, our original contribution is the second statement.

 The statement that the whole orbit $G^\mathbb C\cdot x$
flows to one $U$-orbit $U\cdot z$ is plausible, but is not contained in
Kirwan's work \cite{Kirwan}.  It is a finer statement than the $\GC$-invariance of Kirwan's stratification of $\cpvc$. There are two problems to be aware of: first, for $g\in G^\mathbb C$, $\omega(gx)$ might be a set with more than one point and, second, there is no reason to expect that $\omega(gx)$ lies entirely in the orbit $G^\mathbb C\cdot x$.
  This
proof is just for the complex setting of our complex group $\GC
\actingon \cpvc$.  This is the setting of Kirwan and Ness.

\textit{Remark.}  It has been pointed out to me by Jorge Lauret that the set $\omega(x)$ consists of a single point, see section 2.5 of \cite{Sjamaar:ConvexityofMomentMapping}.  However, our proof does not require the use of this fact.\\

\textbf{Proof of \ref{complex gradient flow thm}}
Consider an orbit $\GC \cdot y$ which is distinguished and let $z
\in \GC\cdot y$ be a critical point.  Let $x$ be any point in $\GC
\cdot y$.  To show that $\omega (x) \subseteq U\cdot z$, we will first show
that the limit set $\omega (x)$ intersects $U\cdot z$ and then show
containment.  First we need to recall some results from Kirwan's
work \cite{Kirwan}.

We have a smooth stratification of $\cpvc$ into strata $S_\beta$
which are $\GC$-invariant.  The strata are determined by a certain decomposition of the critical set $\mathfrak C$ of $||\mu^*||^2$ in $\cpvc$.  This critical set is a finite union $\mathfrak C = \bigcup _{\beta \in B} C_\beta$ where $||\mu^*||^2$ takes a constant value on $C_\beta$ and each $C_\beta$ is $U$-invariant.  We will denote this constant value of $||\mu^*||^2$ on $C_\beta$ by $M_\beta =||\beta||^2$; here $B$ is actually a finite set in $\gc$ and the norm $||\cdot||$ comes from the prescribed inner product on $\gc$.

For $\beta \in B$, the stratum $S_\beta$ is defined to be the set of points which flow via the negative gradient
flow to the critical set $C_\beta$, that is, $S_\beta=\{ x\in \cpvc | \ \omega (x) \subseteq C_\beta  \}$.  In particular, $C_\beta \subseteq S_\beta$.  See section 2 of \cite{Kirwan} for a detailed discussion of this Morse Theory approach to Geometric Invariant Theory.  If $\GC\cdot y \cap C_\beta
\neq \emptyset$ then
$$ \GC\cdot y \cap C_\beta = U\cdot z$$
for $z \in C_\beta$, that is, the critical points in a $G^\mathbb C$-orbit comprise a single $U$-orbit, see \cite[Theorem 6.2]{Ness}.  We show two things.  First, if $x\in G^\mathbb C \cdot z$ is in a neighborhood of $U\cdot z$, then $\omega (x) \subseteq U\cdot z$.  Second, this neighborhood of $U\cdot z$ in $G^\mathbb C\cdot z$ should be the entire orbit; that is, $\omega (x) \subseteq U\cdot z$ for all $x\in G^\mathbb C \cdot z$.  The first is a little more obvious but does rely on
the fact that our moment map is a minimally degenerate Morse
function, see definition 10.1 of \cite{Kirwan}.  That fact that $||\mu^*||^2$ is a minimally degenerate Morse function can be found in section 4 of \cite{Kirwan}.

Fix $\beta$.  We will be interested in $z\in C_\beta$ and the orbit $G^\mathbb C \cdot z$. We define $\mathcal O_\varepsilon = \{ x\in \cpvc \ | \ ||\mu^*||^{2}(x) \in [\ M_\beta,M_\beta+\varepsilon) \} \cap S_\beta$.  This is an open subset of $S_\beta$ that contains $C_\beta = \{ x\in S_\beta \ | \ ||\mu^*(x)||^2= M_\beta \}$.  We observe that $\mathcal O_\varepsilon$ is invariant under the forward flow $\varphi _t$ of $-grad \ \nmus $ as $||\mu^*||^2$ decreases along the trajectories $t\to \varphi_t (x)$.  Since $\GC \cdot z$ is a submanifold of $\cpvc$, hence also of $S_\beta$, $\mathcal O_\varepsilon \cap \GC \cdot z$ is open in $\GC \cdot z$ and contains $U\cdot z$ as $C_\beta$ is $U$-invariant.

\begin{defin} We define $\{ V_{\varepsilon,i} \}$ to be the collection of connected components of $\mathcal O_\varepsilon \cap \GC \cdot z$ that intersect $U\cdot z$.  We define $V_\varepsilon:= \displaystyle \bigcup_i V_{\varepsilon,i}$. \end{defin}

\textit{Remark}.  $V_\varepsilon$ is an open set of $G^\mathbb C\cdot z$ that contains $U\cdot z$. As $U$ has finitely many components, $U=\displaystyle \bigcup_{i=1}^m \phi_i U_0 $ and we can write $V_\varepsilon =\displaystyle \bigcup_{i=1}^m V_{\varepsilon, i}$ where $\phi_i U_0(z) \subseteq V_{\varepsilon,i}$.  The $V_{\varepsilon,i} $ are connected and open in $\GC \cdot z$ as $\mathcal O_\varepsilon \cap \GC\cdot z$ is open in $\GC \cdot z$ and $\GC \cdot z$ is locally connected, see \cite[Theorem 25.3]{Munkrees}. Moreover, since $\mathcal O_\varepsilon$ and $\GC \cdot z$ are invariant under $\varphi_t, t>0$, we see that the components $V_{\varepsilon,i}$ are invariant under forward flow, as well.

\begin{prop}\label{prop: V epsilon in GC for small epsilon} There exists $\varepsilon >0$ such that $\overline V_\varepsilon \subseteq \GC \cdot z$.  Moreover, $\omega (V_\varepsilon) = U\cdot z$ for small $\varepsilon >0$.\end{prop}

\begin{proof}  Before proving this statement, we will show that there exists some open set $A$ containing $U\cdot z$ in $\GC \cdot z$ such that $\overline A$ is a compact subset of $\GC \cdot z$.  Then we will show that $V_\varepsilon \subseteq A$ for small $\varepsilon$.  This would then prove the first assertion of the proposition.

Recall that $G^\mathbb C = U \ exp(iLU)$.  If we let $B=$ the open unit ball in $iLU$ then $A=U\ exp(B)\cdot z$ has the said property, that is, $\overline A$ is a compact subset of $G^\mathbb C \cdot z$.

\begin{lemma} Either $V_\varepsilon \subseteq A$ or $V_\varepsilon \cap \partial A \neq \emptyset$.  For small $\varepsilon > 0$, $V_\varepsilon \subseteq A$. \end{lemma}

This will follow from
\begin{lemma} Either $V_{\varepsilon, i} \subseteq A$ or $V_{\varepsilon, i}\cap \partial A \neq \emptyset$. \end{lemma}

To prove this lemma, suppose $V_{\varepsilon,i} \not \subseteq A$ and $V_{\varepsilon, i} \cap \partial A = \emptyset$.  Since $V_{\varepsilon,i}\cap A$ intersects $U\cdot z$, we see that  $V_{\varepsilon,i} = (V_{\varepsilon,i} \cap A) \cup (V_{\varepsilon,i} \backslash \overline A)$; that is, $V_{\varepsilon,i}$ is separated by these disjoint open sets.  This contradicts the connectedness of $V_{\varepsilon,i}$ and the lemma is proven.

We continue with the proof of the first lemma.  Suppose $V_\varepsilon \not \subseteq A$ for every $\varepsilon >0$.  Then for each $\varepsilon$ there exists some point $p_\varepsilon \in V_\varepsilon \cap \partial A$.  By definition $||\mu^*||^2(p_\varepsilon) \leq M_\beta+\varepsilon$. Letting epsilon go to zero we can find a limit point $p_\infty \in \partial A$ as $\partial A$ is compact.  Hence, $p_\infty \in G^\mathbb C\cdot z -A \subseteq G^\mathbb C \cdot z - U\cdot z$.  Moreover, $||\mu^*||^2(p_\infty) = M_\beta$ and we have found a point in $\GC \cdot z$ which is not on $U\cdot z$ but minimizes $||\mu^*||$ on $G^\mathbb C\cdot z$.  This is a contradiction since $G^\mathbb C \cdot z \cap C_\beta = U\cdot z$ by \cite[Theorem 6.2]{Ness}.  Therefore, $V_\varepsilon \subseteq A$ for small $\varepsilon$.  This proves the first lemma and the first claim in the proposition.

To finish the proof of the proposition, we observe that $U\cdot z= \omega (U\cdot z)\subseteq \omega(V_\varepsilon)$ since $U\cdot z \subseteq C_\beta$ and $\varphi_t$ fixes the points of $C_\beta$ for all $t$.  Thus we just need to show containment in the other direction.  Since the set $V_\varepsilon$ is invariant under forward flow and $V_\varepsilon \subseteq G^\mathbb C \cdot z \subseteq S_\beta$, we see that $\omega (V_\varepsilon) \subseteq \overline{V_\varepsilon}\cap C_\beta \subseteq G^\mathbb C \cdot z\cap C_\beta = U\cdot z$.
\end{proof}

\begin{defin} Let $\mathcal O = \{ x\in G^\mathbb C\cdot z \ | \ \omega (x) \subseteq U\cdot z\}$. \end{defin}

\begin{lemma}Consider the set $\mathcal O$ defined above.  Then $\mathcal O = G^\mathbb C \cdot z$.
\end{lemma}
To prove the lemma it suffices to show that $\mathcal O$ is open and closed in $G^\mathbb C\cdot z$ and intersects each component of $G^\mathbb C \cdot z$.  To see that $\mathcal O$ intersects each component of $G^\mathbb C \cdot z$, we observe that $\mathcal O$ contains $U\cdot z$ and that each component of $G^\mathbb C$ intersects $U$ since $G^\mathbb C = UQ$ and $Q=exp(\q)$ is contractible, see the remarks before Proposition \ref{prop: extend inner product on V to VC}.  Choose $\varepsilon >0$ as in Proposition \ref{prop: V epsilon in GC for small epsilon}.\\

\noindent \textbf {$\mathcal O$ is open}:

We know for small $\varepsilon>0$, $V_\varepsilon$  is open in $\GC \cdot z$, contains $U\cdot z$, and $V_\varepsilon$ is contained in $\mathcal O$ by Proposition \ref{prop: V epsilon in GC for small epsilon}.  It suffices to consider $x\in
\mathcal O \backslash U\cdot z$.  Then there exists $t_*>0$ such that
$\varphi_{t_*}(x)$ intersects $V_\varepsilon$, from the definition of $\mathcal O$.
But $\varphi_{-t_*}:V_\varepsilon \to \varphi_{-t_*}(V_\varepsilon)$ is a diffeomorphism of $\GC \cdot z$ (and also of $S_\beta$). 
  Thus, $\varphi_{-t_*} (V_\varepsilon)$ is
an open set in $\GC\cdot z$ containing $x$, which is contained in
$\mathcal O$.  Therefore $\mathcal O$ is open.\\

\noindent \textbf{$\mathcal O$ is closed}:

We will show $\partial \mathcal O = \emptyset$; here we mean the boundary of $\mathcal O$ in the topological space $G^\mathbb C \cdot z$.  Take $y_n \in \mathcal O$ such that $y_n \to y \in G^\mathbb C\cdot z$.  Since $z\in C_\beta\subseteq S_\beta$ and $S_\beta$ is $G^\mathbb C$-invariant, it follows that $y\in G^\mathbb C\cdot z \subseteq S_\beta$ and hence $\omega (y) \subseteq C_\beta$.  Thus, there exists $M>0$ such that $\varphi _M(y) \in \mathcal O _\varepsilon$.  We will denote the component of $\mathcal O_\varepsilon \cap \GC \cdot z$ containing $\varphi_M(y) $ by $\mathcal O_\varepsilon ^y$; again, this component is open in $\GC\cdot z$ as $\GC \cdot z$ is locally connected.  Observe that for $t\geq M$, $\varphi_t(y)\in \mathcal O_\varepsilon ^y$ and $\varphi_s (\mathcal O_\varepsilon ^y ) \subseteq \mathcal O_\varepsilon ^y$ for $s\geq 0$ as $\varphi_s$ leaves $\mathcal O_\varepsilon \cap G^\mathbb C \cdot z$ invariant for $s\geq 0$. Since $\varphi_t$ is a diffeomorphism on $S_\beta$ which preserves $\GC \cdot z$, $\varphi_M^{-1} (\mathcal O_\varepsilon ^y)$ is an open set of $\GC \cdot z$ containing y.

We assert that $\mathcal O_\varepsilon ^y \cap V_\varepsilon \neq \emptyset$.  Since $y_n\in \mathcal O$, we know there exists $T_n >0$ such that $\varphi_{T_n} (y_n) \in V_\varepsilon$, by definition of $\mathcal O$.  Additionally, for $t\geq T_n$, $\varphi_t(y_n)\in V_\varepsilon$ by the flow invariance of $V_\varepsilon$.

Pick $N$ such that $y_N \in \varphi_M^{-1}(\mathcal O_\varepsilon ^y)$, which we can do as $\varphi_M^{-1}(\mathcal O_\varepsilon ^y)$ is open and $y_n \to y$.  Then we have $\varphi_M(y_N) \in \mathcal O _\varepsilon ^y$, a single component of $\mathcal O_\varepsilon \cap \GC \cdot z$, and $\varphi_{T_N}(y_N)\in V_\varepsilon$.

\begin{quote}
    If $M\geq T_N$, then $\varphi_M(y_N) = \varphi_{M-T_N}( \varphi_{T_N}(y_N)) \in \varphi_{M-T_N} (V_\varepsilon) \subseteq V_\varepsilon$.

    That is, $\varphi_M(y_N) \in \mathcal O_\varepsilon ^y \cap V_\varepsilon \neq \emptyset$.\\

    If $T_N\geq M$, then $\varphi_{T_N}(y_N) = \varphi_{T_N-M}( \varphi_M(y_N)) \in \varphi_{T_N-M} (\mathcal O_\varepsilon ^y) \subseteq \mathcal O_\varepsilon ^y$.

    That is, $\varphi_{T_N}(y_N)\in \mathcal O_\varepsilon ^y \cap V_\varepsilon \neq \emptyset$.
\end{quote}

Thus, $\mathcal O_\varepsilon ^y$ being a connected component of $\mathcal O_\varepsilon \cap \GC \cdot z$ which intersects $V_\varepsilon$, a union of connected components of $\mathcal O_\varepsilon \cap G^\mathbb C\cdot z $, we have $\mathcal O_\varepsilon^y \subseteq V_\varepsilon$.  That is, $y\in \mathcal O$  since $\varphi_t (y) \in V_\varepsilon$ for $t\geq M$ and $\omega (V_\varepsilon) \subseteq U\cdot z$ by Proposition \ref{prop: V epsilon in GC for small epsilon}.  This proves the lemma. \\

This completes the proof of theorem \ref{complex gradient flow thm}.

\section{Applications to the left-invariant geometry of Lie groups}
We apply the previous results to the study of left-invariant metrics on nilpotent and solvable Lie groups.  The relationship between left-invariant Einstein metrics on solvable Lie groups, left-invariant Ricci soliton metrics on nilpotent Lie groups, and Geometric Invariant Theory was established and explored by J. Heber \cite{Heber} and J. Lauret \cite{LauretNilsoliton}.  We present a sketch of this relationship below and refer the reader to \cite{Lauret:EinsteinSolvandNilsolitonsCordobaConf2007} for more details.

In the search for left-invariant Einstein metrics on solvable Lie groups, it is sufficient to consider the nilradical of the given solvable Lie group.

\begin{thm}[Lauret \cite{LauretStandard}] Let $S$ be a solvable Lie group with nilradical $N$.  Then $S$ admits a left-invariant Einstein metric if and only if $N$ admits a left-invariant Ricci soliton metric.\end{thm}

Left-invariant Ricci soliton metrics on nilpotent Lie groups are often referred to as \textit{nilsoliton} metrics.  The search for, and classification of, nilsoliton metrics aids in the search for, and classification of, left-invariant Einstein metrics on solvmanifolds; however, these metrics are very interesting in their own right.\\

The relationship between left-invariant metrics on nilpotent Lie groups and Geometric Invariant Theory is as follows.  Consider a nilpotent Lie group $N$ with Lie algebra $\N$.  A left-invariant metric on $N$ corresponds to a choice of inner product on $\N$.  Thus, the space of left-invariant metrics on $N$ is the space of inner products on $\N$.  We can vary the inner products on $\N$ to search for nilsolitons, or we can fix our choice of inner product on $\N$ and instead vary the Lie algebra structure on $\N$.  This is the perspective taken by Lauret.

Consider the vector space $\mathbb R^n$ with the usual inner product; that is, so that standard basis is orthonormal.  We consider the space of skew-symmetric, bilinear forms on $\mathbb R^n$
    $$ V= \wedge ^2(\mathbb R^n)^* \otimes \mathbb R^n = \{ \mu : \mathbb R^n \times \mathbb R^n \to \mathbb R^n \ | \ \mu \mbox{ is bilinear and skew-symmetric } \}$$
The set of Lie algebra brackets is clearly a subset of the vector space above; in fact, the set of Lie algebra structures is a variety.  Moreover, the set of nilpotent Lie algebra brackets is also a  variety.  It is  described by the polynomials describing the Jacobi identity and nilpotency (via Cartan's criterion for nilpotentcy).

There is a natural $GL_n\mathbb R$ action on $V$ which preserves the varieties of Lie algebra structures.  For $\mu \in V$, $g\in GL_n\mathbb R$, and $X,Y\in \mathbb R^n$ we have
    $$g\cdot \mu (X,Y) = g\mu(g^{-1}X,g^{-1}Y)$$
In this setting, the $GL_n\mathbb R$-orbits are precisely the isomorphism classes of Lie algebra structures on $\mathbb R^n$.

The inner product on $\mathbb R^n$ extends naturally to an inner product on $V= \wedge ^2(\mathbb R^n)^* \otimes \mathbb R^n$ as follows.  Denote the inner product on $\mathbb R^n$ by $\langle \cdot , \cdot \rangle$ and denote its extension to $V$ by the same notation.  Then for $\mu , \lambda \in V$ we define $\langle \mu , \lambda \rangle = \sum_{ij} \langle \mu(e_i,e_j),\lambda(e_i,e_j)\rangle $ where $\{e_i\}$ is an orthonormal basis of $\mathbb R^n$.  On the Lie algebra $\mathfrak{gl}_n\mathbb R$ we use the usual inner product from the trace form, that is,  $\langle \langle X , Y \rangle \rangle = tr(X Y^t)$ for $X,Y \in \mathfrak{gl}_n\mathbb R$.

Using these inner products, we can construct the moment map $\tilde m$ for the action of $GL_n\mathbb R$ on $V$ and similarly $m$ for the action of $GL_n\mathbb R$ on $\mathbb PV$ (see section 2).

\begin{thm}[Lauret] Let $N_\mu$ denote the simply connected nilpotent Lie group with left-invariant metric whose Lie algebra $\N_\mu$ (with inner product) corresponds to the point $\mu \in V$.  Then $N_\mu$ is a nilsoliton if and only if $\mu$ is a critical point of $F(v)=||m\circ \pi||^2(v)$.  Equivalently, $N_\mu$ is an Einstein nilradical if and only if the orbit $GL_n\mathbb R \cdot \mu$ is distinguished.
\end{thm}
This theorem can be found in \cite{Lauret:EinsteinSolvandNilsolitonsCordobaConf2007}.  The last equivalence is not stated using the label of distinguished orbit but is stated using the idea.  In section 10 of \cite{Lauret:EinsteinSolvandNilsolitonsCordobaConf2007} there are several open questions of interest which are presented.  We state Question \# 5 from this list.

\begin{question}  Consider the function $F: V \to \mathbb R$ defined by $F(v) = ||m \circ \pi(v)||^2$ where $\pi : V\to \mathbb PV$ is the usual projection and $m$ is the moment map on real projective space.  Define $\mu_t$ to be the integral curve of $-grad\ F$ starting at $\mu_0$ on the sphere of radius 2.  Is $\mu_\infty$ (the limit point along the integral curve) contained in the orbit $GL_n\mathbb R \cdot \mu_0$ if $N_{\mu_0}$ is an Einstein nilradical?
\end{question}

This is clearly a special case of our work (Theorem \ref{thm: real gradient flow thm}).  The flow restricted to the sphere of radius 2 in $V$ projects onto the flow in projective space; recall that the sphere is a 2-to-1 cover of $\mathbb PV$.  Thus, convergence within the group orbit in the sphere is equivalent to convergence within the group orbit in projective space.  Finally, as $N_{\mu_0}$ being an Einstein nilradical is equivalent to the orbit $GL_n\mathbb R \cdot \mu_0$ begin distinguished, we have the following.

\begin{thm} Let $N_{\mu_0}$ be an Einstein nilradical.  Let $\mu_\infty$ denote the limit point of the negative gradient flow of the function $F$ starting at $\mu_0$.  Then $\mu_\infty$ is contained in the orbit $GL_n\mathbb R \cdot \mu_0$; that is, $N_{\mu_0}$ and $N_{\mu_\infty}$ are isomorphic Lie groups.
\end{thm}

Lastly, we apply Corollary \ref{cor: real orbits which are complex equiv} to the setting of real forms of complex Lie algebras to obtain another interesting geometric consequence of our work.  Let $N$ be a simply connected real nilpotent Lie group with Lie algebra $\N$. Let $N^\mathbb C$ denote the simply connected complex Lie group with Lie algebra $\N^\mathbb C = \N \otimes \mathbb C$.  We call $N^\mathbb C$ the complexification of $N$.

\begin{thm}\label{thm: real forms of same complex group}   Let $N_1$ and $N_2$ be two real simply connected nilpotent Lie groups whose complexifications $N_1^\mathbb C$, $N_2^\mathbb C$ are isomorphic.  Then $N_1$ is an Einstein nilradical if and only if $N_2$ is an Einstein nilradical.
\end{thm}

\textit{Remark}.  This theorem has also been obtained by Nikolayevsky in \cite[Theorem 6]{Nikolayevsky:EinsteinSolvmanifoldsandPreEinsteinDerivation} where he studies closed orbits of a particular reductive group associated to each nilmanifold.  There the philosophy of comparing real and complex group orbits is also employed.

\providecommand{\bysame}{\leavevmode\hbox to3em{\hrulefill}\thinspace}
\providecommand{\MR}{\relax\ifhmode\unskip\space\fi MR }
\providecommand{\MRhref}[2]{%
  \href{http://www.ams.org/mathscinet-getitem?mr=#1}{#2}
}
\providecommand{\href}[2]{#2}

\end{document}